\documentclass[a4paper,12pt]{amsart}
\usepackage[utf8]{inputenc}

\usepackage{amsmath,amssymb,amsthm}
\usepackage{bbm}
\usepackage{bm}
\usepackage{xspace}
\usepackage{csquotes}
\usepackage{standalone}

\usepackage[x11names,table]{xcolor}
\usepackage[naturalnames=false]{hyperref}
\usepackage{graphicx}
\graphicspath{ {img/} }
\usepackage{tikz}
\usetikzlibrary{calc}
\usetikzlibrary{cd}
\usepackage{csquotes}
\usepackage{booktabs}
\hypersetup{
	colorlinks = true,          
	urlcolor = DeepSkyBlue3, linkcolor = Chartreuse4, citecolor = DarkOrange2
}
\usepackage{fullpage}
\RequirePackage[backend=bibtex,isbn=false,url=false,maxbibnames=5,bibencoding=utf8]{biblatex}
\renewbibmacro{in:}{, }
\DeclareFieldFormat[misc]{title}{\mkbibquote{#1}}
\DeclareFieldFormat[article]{volume}{\mkbibbold{#1}}
\DeclareFieldFormat{url}{\textsc{url}: \href{#1}{#1}}
\DeclareFieldFormat{doi}{\textsc{doi}: \href{http://dx.doi.org/#1}{#1}}
\DeclareFieldFormat{eprint:arxiv}{arXiv:\href{https://arxiv.org/abs/#1}{#1}}
\newcommand\RR{{\mathbb R}}

\newcommand\TT{{\mathbb T}}
\newcommand\ZZ{{\mathbb Z}}

\newcommand\cP{{\mathcal P}}

\newcommand\cH{{\mathcal H}}

\newcommand\cT{{\mathcal T}}

\newcommand\sX{{\mathsf X}}
\newcommand\sY{{\mathsf Y}}

\newcommand\0{{\mathbf 0}}
\newcommand\1{{\mathbf 1}}

\newcommand\torus[1]{\RR^{#1}/\RR\1} 

\DeclareMathOperator\conv{conv}
\DeclareMathOperator\tconv{tconv}

\DeclareMathOperator\argmin{arg\,min}

\DeclareMathOperator\proj{proj}

\newcommand\trop{\text{trop}}

\renewcommand{\vec}[1]{\mathbf{#1}}

\newtheorem{theorem}{Theorem}
\newtheorem{lemma}[theorem]{Lemma}
\newtheorem{proposition}[theorem]{Proposition}
\newtheorem{corollary}[theorem]{Corollary}

\theoremstyle{definition}
\newtheorem{definition}[theorem]{Definition}
\newtheorem{example}[theorem]{Example}
\theoremstyle{remark}

\newtheorem{remark}[theorem]{Remark}

\definecolor{tolorange}{RGB}{238,119,51}  
\definecolor{tolblue}{RGB}{0,119,187}     
\definecolor{tolgreen}{RGB}{0,153,136}    
\definecolor{tolpurple}{RGB}{136,34,85}   
\definecolor{ibmyellow}{RGB}{255, 176, 0} 
\definecolor{darkgray}{RGB}{64,64,64} 

\tikzstyle{blackdot} = [circle,draw=black,fill=black,scale=0.5]

\usepackage{expl3}
\ExplSyntaxOn
\newcommand \breakDOI[1]
 {
  \tl_map_inline:nn { #1 } { \href{http://dx.doi.org/#1}{##1} \penalty0 \scan_stop: }
 }
\ExplSyntaxOff

\DeclareFieldFormat{doi}{%
  \mkbibacro{DOI}\addcolon\space
  {\breakDOI{#1}}}

 \usepackage[colorinlistoftodos,bordercolor=orange,backgroundcolor=orange!20,linecolor=orange,textsize=tiny]{todonotes}

\title{Tropical convexity in location problems}

\author{Andrei Com\u{a}neci}

\address{
	Institut f\"{u}r Mathematik,
    Technische Universit\"{a}t Berlin \\
	\texttt{comaneci@math.tu-berlin.de}	
}



\begin{document}

\begin{abstract}
  We investigate location problems whose optimum lies in the tropical convex hull of the input points.
  Firstly, we study geodesically star-convex sets under the asymmetric tropical distance and introduce the class of tropically quasiconvex functions whose sub-level sets have this shape.
  The latter are related to monotonic functions.
  Then we show that location problems whose distances are measured by tropically quasiconvex functions as before give an optimum in the tropical convex hull of the input points.
  We also show that a similar result holds if we replace the input points by tropically convex sets.
  Finally, we focus on applications to phylogenetics presenting properties of consensus methods arising from our class of location problems.
\end{abstract}

\maketitle

\section{Introduction}

There is a recent interest in studying location problems in tropical geometry, especially in the use of tropical methods to data analysis.
Maybe the first article to promote such problems with a view towards ``tropical statistics'' is the work of Lin et al. \cite{LSTY:2017}.
They showed that tropical convexity in tree spaces has some better properties than the geometry of Billera, Holmes, and Vogtmann (BHV) \cite{BHV}.
This encouraged them to propose location estimators based on the symmetric tropical distance that could potentially exploit tropical convexity.
In particular, this would give a tropical approach to the consensus problem from phylogenetics \cite{Bryant:2003}.

The connection for the proposed location statistics to tropical convexity was not well understood.
For example, they noticed that tropical Fermat--Weber points can lie outside the tropical convex hull of the input points \cite[Example~26]{LSTY:2017}, although it was found later that one can find Fermat--Weber points inside the tropical convex hull \cite[Lemma~3.5]{Page+Yoshida+Zhang:2020}.
However, the unclear connection makes it difficult to obtain solutions that can be interpreted in the phylogenetic setting; see also~\cite{Lin+Yoshida:2018}.

Recently, we could show that studying the Fermat--Weber problem using an \emph{asymmetric} distance function leads to a better explanation in terms of tropical convexity~\cite{TropMedian}.
In particular, it provides a clear approach based on tropical convexity to the consensus problem from phylogenetics.
Moreover, various desirable properties of consensus methods were obtained by exploiting tropical convexity.
In fact, the good properties were solely due to tropical convexity and not the particular distance function which motivates the search for other methods with similar properties.


In this paper, we focus on location problems that have the potential of exploiting tropical convexity.
More specifically, we care of those location estimators that will belong to the tropical convex hull of the input points.
Such estimators are based on distances that reflect the tropical structure of the space and can be seen as a counterpart to similar studies regarding location problems and ordinary convexity.


Significant work was done for understanding geometric properties of location problems and their relationship to ordinary convexity.
The case of Chebyshev centers dates back to the 60s in the work of Garkavi~\cite{Garkavi:1964} and Klee~\cite{Klee:1960}.
More general location problems in a normed space were studied by Wendell and Hurter~\cite{Wendell+Hurter:1973}, while a focus on geometric properties of Fermat--Weber problems with varying distances is covered by Durier and Michelot~\cite{Durier+Michelot:1985}.
What is more, it was shown that finding an optimal solution in the (ordinary) convex hull for every set of points is equivalent to having an inner product space in three dimensions or more; a general form of this result was obtained by Durier~\cite{Durier:1994,Durier:1997}.




The results mentioned above show a strong relationship between ordinary convexity and a Euclidean structure.
Tropical convexity, on the other hand, it is related to the lattice structure of $(\RR^n,\leq)$.
Hence, we have to focus on ``monotonic'' distances.
To interpret geometrically monotonic functions in the quotient space $\torus{n}$, we notice that all sub-level sets share a similarity: they are geodesically star-convex with respect to the asymmetric tropical distance.
The latter can be seen by remarking that geodesic segments are images of order segments in $(\RR^n,\leq)$.
The resulting sets, called $\triangle$-star-convex, and functions, called $\triangle$-star-quasiconvex, are discussed in sections~\ref{sec:starConv} and~\ref{sec:tropQuasiconv}, respectively.

In section~\ref{sec:locProb} we focus on location problems in which distances to the sites are measured by $\triangle$-star-quasiconvex functions.
We show that this setting guarantees optimal locations in the tropical convex hull of the input.
We will see that the triangle inequality does not play any role, which emphasizes the differences between tropical and ordinary convexity.
Further, this setting allows for very general location problems where dissimilarities are not necessarily distances; triangle inequality is generally assumed in location science when dealing with geographic location~\cite{LocationScience:2019}, but it is not reasonable for more general data~\cite{Tversky:1977} and never assumed in the construction of M-estimators~\cite[\textsection 3.2]{Huber:1981}.



We have further a few examples of location problems from the literature that end in our setting.
In particular, location problems involving the symmetric and asymmetric tropical distances.
However, the former case might contain cases where some optima are outside the tropical convex hull of the input.
So what is the precise distinction between the symmetric and the asymmetric tropical distances that causes the above behaviour?
We show that \emph{strict} $\triangle$-star-convexity is the answer.
This motivates that study of regularized versions discussed in \textsection\ref{subsec:regular}.

We briefly show in section~\ref{sec:locProbTropConvSites} that we can extend the results to the case when the sites are tropically convex sets.
Then section~\ref{sec:tropConsensus} deals with the main application to phylogenetics: the tropical approach to consensus methods.
Our general setting provides a large class of tropically convex consensus methods as defined in \cite[\textsection 5]{TropMedian}.
Furthermore, we enlarge the list of desirable properties of these consensus methods that were given in the previously cited work.
Finally, we conclude with section~\ref{sec:conclusion} consisting of highlights and possible directions for future research.

\section{Tropical convexity} \label{sec:tropConv}


The purpose of this section is to fix the notation and emphasize the basic properties of tropically convex sets that will be used later.
One can consult the book of Joswig~\cite{ETC} for more details.
We will use both semirings $\TT^{\min}=\left(\RR\cup\{\infty\},\wedge,+\right)$ and $\TT^{\max}=\left(\RR\cup\{-\infty\},\vee,+\right)$ where $x\wedge y=\min(x,y)$ and $x\vee y=\max(x,y)$.
They are isomorphic under the map $x\mapsto -x$, but it is better to be seen as dual to each other.
This duality will play an important role later similar to the relationship between $\max$-tropical polytopes and $\min$-tropical hyperplanes \cite[Chapter~6]{ETC}.



Since our applications deal with points of finite entry, we will define tropical geometric objects in $\RR^n$ and $\torus{n}$.
It also exploits the common set of $\TT^{\max}$ and $\TT^{\min}$ and we can make use of the vector space structure. 

A \emph{$\min$-tropical cone} $K\subset\RR^n$ is a set closed under $\min$-tropical linear combinations: $(x+\lambda\1)\wedge (y+\mu\1)\in K$ for all $x,y\in K$ and $\lambda,\mu\in\RR$.
The image of a $\min$-tropical cone in $\torus{n}$ is called a \emph{$\min$-tropically convex set}.
A common example is the $\min$-tropical hyperplane with apex $v$ which is the set $H^{\min}_v=\left\{x\in\torus{n}:|\argmin_j(x_j-v_j)|\geq 2\right\}$.
The $\max$-tropical cones and $\max$-tropically convex sets are defined similarly, replacing $\min$ by $\max$ in the previous definitions.
One can also see them as images of $\min$-tropical cones and $\min$-tropically convex sets under $x\mapsto -x$.

The $\min$-tropical convex hull of two points $a,b\in\torus{n}$ will be denoted by $[a,b]_{\max}$ and is called the \emph{$\min$-tropical segment} between $a$ and $b$.
We will also use the notation $(a,b)_{\min}=[a,b]_{\min}\setminus\{a,b\}$ for the \emph{open} $\min$-tropical segment between $a$ and $b$.
Similarly, we define $[a,b]_{\max}$ and $(a,b)_{\max}$.


The \emph{$\min$-tropical convex hull} of a set $A\subset\torus{n}$ is the smallest $\min$-tropically convex set containing $A$ and we denote it by $\tconv^{\min}(A)$.
It can be related to the $\max$-tropical semiring by \cite[Proposition~5.37]{ETC}.
For this we need to introduce the $\max$-tropical sector $S_i^{\max}=\{x\in\torus{n}:x_i\geq x_j\ \forall j\in[n]\}=\{x\in\torus{n}:i\in\argmin_j x_j\}$.
Then \cite[Proposition~5.37]{ETC} says that $x$ belongs to $\tconv^{\min}(A)$ if and only if for each $i\in[n]$ there exists $a_i\in A$ such that $x\in a_i+S_i^{\max}$.
For the case of $\max$-tropically convex hull just reverse $\min$ with $\max$.

We say that a point $a$ of a $\min$-tropically convex set $A$ is \emph{$i$-exposed} if $\left(a+S_i^{\min}\right)\cap A=\{a\}$.
If a point is $i$-exposed for some $i\in[n]$, then we simply call it \emph{exposed}.




Since the order $\leq$ on $\RR^n$ is strongly related to tropical convexity, we will focus on monotonic function.
We say that a function $f:X\to\RR$, defined on a subset $X$ of $\RR^n$, is \emph{increasing} if for every $x, y\in X$ with $x\leq y$ we have $f(x)\leq f(y)$.
We call $f$ \emph{strictly increasing} if $f(x)<f(y)$ whenever $x\leq y$ and $x\neq y$. 

For $a,b\in\RR^n$ and $a\leq b$, we denote by $[a,b]_{\leq}$ the set of points $x\in\RR^n$ such that $a\leq x\leq b$ and call it the \emph{order segment} between $a$ and $b$.
It can also be written as a box: $[a,b]_{\leq}=[a_1,b_1]\times\dots\times[a_n,b_n]$.
Its image in $\torus{n}$ is a \emph{polytrope}, i.e. it is both $\min$- and $\max$-tropically convex \cite[\textsection 6.5]{ETC}, which we call a \emph{box polytrope}.
A particular case is presented in the following example.



Consider the asymmetric distance $d_\triangle(a,b)=\sum_i(b_i-a_i)-n\min_j(b_j-a_j)$ defined on $\torus{n}$ \cite{TropMedian}.
We are interested in geodesic segments under this distance, which are portrayed in Figure~\ref{fig:triangle_seg}.
This is different from the geodesic convexity discussed in \cite[\textsection 5.3]{ETC} which focuses on the symmetric tropical distance.

\begin{definition}
    For two points $a,b\in\torus{n}$ we define the (oriented) \emph{geodesic segment} between $a$ and $b$ under $d_\triangle$ as $[a,b]_\triangle:=\left\{x\in\RR^n:d_\triangle(a,x)+d_\triangle(x,b)=d_\triangle(a,b)\right\}$.
\end{definition}

\begin{remark} \label{rmk:geoSegment}

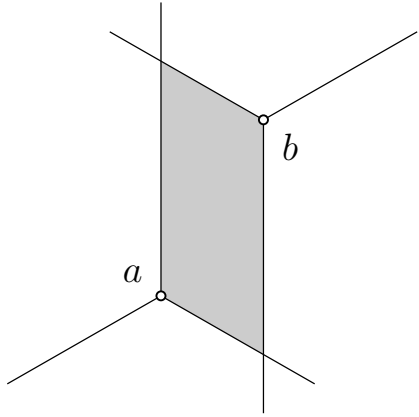
\begin{figure}
    \centering
         \begin{tikzpicture}[scale=1.3,x={(0.866,-0.5)},y={(0,1)},z={(-0.866,-0.5)}]
    \fill[black!20!white] (0,0) -- (1,0) -- (1,2) -- (0,2) -- cycle;

    \draw (0,0) -- (-1.5,-1.5);
    \draw (0,0) -- (1.5,0);
    \draw (0,0) -- (0,2.5);

    \draw (1,2) -- (2.5,3.5);
    \draw (1,2) -- (-0.5,2);
    \draw (1,2) -- (1,-0.5);

    \node at (0,0) [circle, draw= black, fill = white, semithick, inner sep = 1pt,label={[black]170:$a$}] {};
    
    \node at (1,2) [circle, draw= black, fill = white, semithick, inner sep = 1,label={[black]-10:$b$}] {};
\end{tikzpicture}
    \caption{The geodesic segment $[a,b]_\triangle$ is marked with grey}
    \label{fig:triangle_seg}
\end{figure}


The geodesic segment $[a,b]_\triangle$ is a (box) polytrope.
To see this, we point out that $[a,b]_\triangle=(a+S_i^{\min})\cap(b+S_i^{\max})$ where $i$ is any index from $\argmin_j(b_j-a_j)$; the equality can be also seen in Figure~\ref{fig:triangle_seg}.
What is more, if we choose representatives $a$ and $b$ such that $\min_j(b_j-a_j)=0$, then $[a,b]_\triangle$ is the image of $[a,b]_{\leq}$ in $\torus{n}$.
The $\min$-tropical vertices $[a,b]_\triangle$ are of the form $v_j=b-(b_j-b_i+a_i-a_j)e_j=b-\left(b_j-a_j-\min_\ell(b_\ell-a_\ell)\right)e_j$ for $j\in[n]$

\end{remark}

\begin{remark} \label{rmk:basicGeoSeg}
    The set $[a,b]_\triangle$ contains the ordinary segment $[a,b]$ but also the $\min$- and $\max$-tropical segments between $a$ and $b$.
    What is more, for every $c\in\torus{n}$ the $\min$-tropical segment between $a$ and $b$ is contained in $[c,a]_\triangle\cup[c,b]_\triangle$.

    To see the latter statement, we take arbitrary representatives modulo $\RR\1$ for $a$ and $b$ and show that $a\wedge b\in[c,a]_\triangle\cup[c,b]_\triangle$.
    Let $i\in\argmin_j[(a_j\wedge b_j)-c_j]$.
    Without loss of generality, we can assume that $a_i\wedge b_i=a_i$.
    Thus, $a\wedge b\in (c+S_i^{\min})\cap(a+S_i^{\max})=[c,a]_\triangle$.
\end{remark}

The \emph{canonical coordinates} of a point $x\in\torus{n}$ are the entries of the $\widehat x\in\RR^n$ defined by $\widehat x= x-\left(\min_j x_j\right)\1$.
This is a representative of $x$ modulo $\RR\1$ such that all its entries are non-negative and at least one entry is $0$.

\begin{definition}
    We say that $K$ is a \emph{strictly $\min$-tropically convex cone} if $K$ is a $\min$-tropically convex cone and for every $a,b\in K$ such that $a\wedge b$ is different from $a$ and $b$ modulo $\RR\1$, then $a\wedge b$ belongs to the interior of $K$.

    We say that a subset of $\torus{n}$ is \emph{strictly $\min$-tropically convex} if it is the image of a strictly $\min$-tropically convex cone under the canonical projection $\RR^n\to\torus{n}$.
\end{definition}

\begin{remark} \label{rmk:strictTropConvSeg}
    A subset $L$ of $\torus{n}$ is strictly $\min$-tropically convex if all the points of the open $\min$-tropical segment $(a,b)_{\min}$ belong to the interior of $L$, where $a$ and $b$ are distinct points in $L$.
\end{remark}

\begin{proposition} \label{prop:fullDimStrictTropConv}
    Any strictly $\min$-tropically convex set is a singleton or its closure coincides with the closure of its interior.
    Moreover, all of its boundary points are exposed.
\end{proposition}

\begin{proof}
    The first part results from Remark~\ref{rmk:strictTropConvSeg}.
    For the second part, consider $v$ which is not exposed.
    Then there exist $p,q$ in the strictly $\min$-tropically convex set such that $v\in(p,q)_{\min}$.
    According to the same remark, $v$ is an interior point.
\end{proof}

\section{$\triangle$-star-convex sets} \label{sec:starConv}

\begin{definition}
    A \emph{$\triangle$-star-convex} set with kernel $v$ is a non-empty set $K\subseteq\torus{n}$ such that for every point $w\in K$ we have $[v,w]_\triangle\subseteq K$.
    We call $K$ \emph{strictly} $\triangle$-star-convex if $[v,w]_\triangle\setminus\{w\}$ belongs to the interior of $K$ for every $w\in K$.
\end{definition}

Since $[v,w]_\triangle$ contains the ordinary segment $[v,w]$, we conclude that $\triangle$-star-convex sets are also star-convex in the ordinary sense.
We show now that $\triangle$-star-convex sets are $\min$-tropically convex.

\begin{proposition}
    Any $\triangle$-star-convex set is $\min$-tropically convex.
\end{proposition}

\begin{proof}
    Let $K$ be a $\triangle$-star-convex set with kernel $v$ and $a,b$ arbitrary points in $K$.
    According to Remark~\ref{rmk:basicGeoSeg}, we have $[a,b]_{\min}\subseteq [v,a]_\triangle\cup [v,b]_\triangle$.
    The latter set is contained in $K$ due to its $\triangle$-star-convexity.
\end{proof}

However, $\triangle$-star-convex sets might not be $\max$-tropically convex.
For example, the image of the regular simplex $\Delta_n=\conv\{e_1,\dots,e_n\}$ in $\torus{n}$ is $\triangle$-star-convex but not $\max$-tropically convex.

\begin{example}
    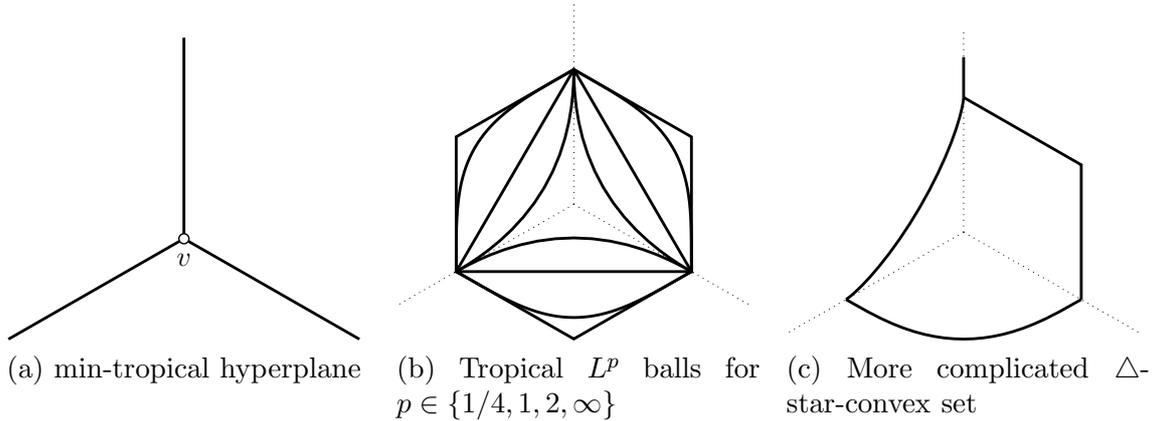
\begin{figure}
    \centering

    \begin{tabular}{p{0.3\textwidth}p{0.3\textwidth}p{0.3\textwidth}}
        \begin{tikzpicture}[scale=2,x={(-0.866cm,-0.5cm)},y={(0cm,1cm)},z={(0.866cm,-0.5cm)}]

        \draw[very thick] (0,0,0) -- (1.5,0,0);
        \draw[very thick] (0,0,0) -- (0,1.5,0);
        \draw[very thick] (0,0,0) -- (0,0,1.5);

        \node (apex) at (0,0,0) [circle, draw= black, fill = white, semithick, inner sep = 1.5pt, label={[black]below:$v$}] {};
    \end{tikzpicture} & \begin{tikzpicture}[scale=2,x={(-0.866cm,-0.5cm)},y={(0cm,1cm)},z={(0.866cm,-0.5cm)}]

        \draw[very thick] (1,0,0) -- (0.999315,0.228771, 0) -- (0.997257,0.323309, 0) -- (0.993825,0.395518, 0) -- (0.989013,0.455973, 0) -- (0.982815,0.508743, 0) -- (0.975221,0.555893, 0) -- (0.96622,0.598638, 0) -- (0.955796,0.637759, 0) -- (0.943931,0.673788, 0) -- (0.930605,0.707107, 0) -- (0.91579,0.737997, 0) -- (0.899454,0.766672, 0) -- (0.881559,0.793297, 0) -- (0.862058,0.818004, 0) -- (0.840896,0.840896, 0) -- (0.818004,0.862058, 0) -- (0.793297,0.881559, 0) -- (0.766672,0.899454, 0) -- (0.737997,0.91579, 0) -- (0.707107,0.930605, 0) -- (0.673788,0.943931, 0) -- (0.637759,0.955796, 0) -- (0.598638,0.96622, 0) -- (0.555893,0.975221, 0) -- (0.508743,0.982815, 0) -- (0.455973,0.989013, 0) -- (0.395518,0.993825, 0) -- (0.323309,0.997257, 0) -- (0.228771,0.999315, 0) -- (0,1,0);
        \draw[very thick] (1,0,0) -- (0.999315, 0, 0.228771) -- (0.997257, 0, 0.323309) -- (0.993825, 0, 0.395518) -- (0.989013, 0, 0.455973) -- (0.982815, 0, 0.508743) -- (0.975221, 0, 0.555893) -- (0.96622, 0, 0.598638) -- (0.955796, 0, 0.637759) -- (0.943931, 0, 0.673788) -- (0.930605, 0, 0.707107) -- (0.91579, 0, 0.737997) -- (0.899454, 0, 0.766672) -- (0.881559, 0, 0.793297) -- (0.862058, 0, 0.818004) -- (0.840896, 0, 0.840896) -- (0.818004, 0, 0.862058) -- (0.793297, 0, 0.881559) -- (0.766672, 0, 0.899454) -- (0.737997, 0, 0.91579) -- (0.707107, 0, 0.930605) -- (0.673788, 0, 0.943931) -- (0.637759, 0, 0.955796) -- (0.598638, 0, 0.96622) -- (0.555893, 0, 0.975221) -- (0.508743, 0, 0.982815) -- (0.455973, 0, 0.989013) -- (0.395518, 0, 0.993825) -- (0.323309, 0, 0.997257) -- (0.228771, 0, 0.999315) -- (0,0,1);
        \draw[very thick] (0,1,0) -- (0, 0.999315,0.228771) -- (0, 0.997257,0.323309) -- (0, 0.993825,0.395518) -- (0, 0.989013,0.455973) -- (0, 0.982815,0.508743) -- (0, 0.975221,0.555893) -- (0, 0.96622,0.598638) -- (0, 0.955796,0.637759) -- (0, 0.943931,0.673788) -- (0, 0.930605,0.707107) -- (0, 0.91579,0.737997) -- (0, 0.899454,0.766672) -- (0, 0.881559,0.793297) -- (0, 0.862058,0.818004) -- (0, 0.840896,0.840896) -- (0, 0.818004,0.862058) -- (0, 0.793297,0.881559) -- (0, 0.766672,0.899454) -- (0, 0.737997,0.91579) -- (0, 0.707107,0.930605) -- (0, 0.673788,0.943931) -- (0, 0.637759,0.955796) -- (0, 0.598638,0.96622) -- (0, 0.555893,0.975221) -- (0, 0.508743,0.982815) -- (0, 0.455973,0.989013) -- (0, 0.395518,0.993825) -- (0, 0.323309,0.997257) -- (0, 0.228771,0.999315) -- (0,0,1);

        \draw[very thick] (1,0,0) -- (0,1,0) -- (0,0,1) -- cycle;

        \draw[very thick] (1,0,0) -- (0.987726,3.78942e-005, 0) -- (0.951655,0.000598866, 0) -- (0.893976,0.00296989, 0) -- (0.818136,0.00911863, 0) -- (0.728553,0.0214466, 0) -- (0.630266,0.0424802, 0) -- (0.528522,0.0745316, 0) -- (0.428381,0.119364, 0) -- (0.334335,0.177901, 0) -- (0.25,0.25, 0) -- (0.177901,0.334335, 0) -- (0.119364,0.428381, 0) -- (0.0745316,0.528522, 0) -- (0.0424802,0.630266, 0) -- (0.0214466,0.728553, 0) -- (0.00911863,0.818136, 0) -- (0.00296989,0.893976, 0) -- (0.000598866,0.951655, 0) -- (3.78942e-005,0.987726, 0) -- (0,1,0);
        \draw[very thick] (1,0,0) -- (0.987726, 0, 3.78942e-005) -- (0.951655, 0, 0.000598866) -- (0.893976, 0, 0.00296989) -- (0.818136, 0, 0.00911863) -- (0.728553, 0, 0.0214466) -- (0.630266, 0, 0.0424802) -- (0.528522, 0, 0.0745316) -- (0.428381, 0, 0.119364) -- (0.334335, 0, 0.177901) -- (0.25, 0, 0.25) -- (0.177901, 0, 0.334335) -- (0.119364, 0, 0.428381) -- (0.0745316, 0, 0.528522) -- (0.0424802, 0, 0.630266) -- (0.0214466, 0, 0.728553) -- (0.00911863, 0, 0.818136) -- (0.00296989, 0, 0.893976) -- (0.000598866, 0, 0.951655) -- (3.78942e-005, 0, 0.987726) -- (0,0,1);
        \draw[very thick] (0,1,0) -- (0, 0.987726,3.78942e-005) -- (0, 0.951655,0.000598866) -- (0, 0.893976,0.00296989) -- (0, 0.818136,0.00911863) -- (0, 0.728553,0.0214466) -- (0, 0.630266,0.0424802) -- (0, 0.528522,0.0745316) -- (0, 0.428381,0.119364) -- (0, 0.334335,0.177901) -- (0, 0.25,0.25) -- (0, 0.177901,0.334335) -- (0, 0.119364,0.428381) -- (0, 0.0745316,0.528522) -- (0, 0.0424802,0.630266) -- (0, 0.0214466,0.728553) -- (0, 0.00911863,0.818136) -- (0, 0.00296989,0.893976) -- (0, 0.000598866,0.951655) -- (0, 3.78942e-005,0.987726) -- (0,0,1);
        
        \draw[very thick] (1,0,0) -- (1,1,0) -- (0,1,0) -- (0,1,1) -- (0,0,1) -- (1,0,1) -- cycle;

        \draw[dotted] (0,0,0) -- (1.5,0,0);
        \draw[dotted] (0,0,0) -- (0,1.5,0);
        \draw[dotted] (0,0,0) -- (0,0,1.5);

    \end{tikzpicture} &
        \begin{tikzpicture}[scale=2,x={(-0.866cm,-0.5cm)},z={(0cm,1cm)},y={(0.866cm,-0.5cm)}]
        \draw[very thick] (1,0,0) -- (0.997944,0.183272, 0) -- (0.991775,0.290328, 0) -- (0.981494,0.379132, 0) -- (0.967099,0.457076, 0) -- (0.948586,0.5271, 0) -- (0.925949,0.590697, 0) -- (0.899174,0.648719, 0) -- (0.868237,0.701691, 0) -- (0.8331,0.749945, 0) -- (0.793701,0.793701, 0) -- (0.749945,0.8331, 0) -- (0.701691,0.868237, 0) -- (0.648719,0.899174, 0) -- (0.590697,0.925949, 0) -- (0.5271,0.948586, 0) -- (0.457076,0.967099, 0) -- (0.379132,0.981494, 0) -- (0.290328,0.991775, 0) -- (0.183272,0.997944, 0) -- (0,1,0);

        \draw[very thick] (1,0,0) -- (0.991218, 0, 0.000694765) -- (0.965225, 0, 0.00498991) -- (0.923066, 0, 0.0156608) -- (0.866428, 0, 0.0348985) -- (0.79755, 0, 0.0642855) -- (0.719122, 0, 0.104745) -- (0.634144, 0, 0.156505) -- (0.545785, 0, 0.219091) -- (0.457225, 0, 0.291348) -- (0.371499, 0, 0.371499) -- (0.291348, 0, 0.457225) -- (0.219091, 0, 0.545785) -- (0.156505, 0, 0.634144) -- (0.104745, 0, 0.719122) -- (0.0642855, 0, 0.79755) -- (0.0348985, 0, 0.866428) -- (0.0156608, 0, 0.923066) -- (0.00498991, 0, 0.965225) -- (0.000694765, 0, 0.991218) -- (0,0,1);

        \draw[very thick] (0,0,1) -- (0,1,1) -- (0,1,0);

        \draw[very thick] (0,0,1) -- (0,0,1.3);

        \draw[dotted] (0,0,0) -- (1.5,0,0);
        \draw[dotted] (0,0,0) -- (0,1.5,0);
        \draw[dotted] (0,0,0) -- (0,0,1.5);

    \end{tikzpicture} \\
        \small (a) $\min$-tropical hyperplane  & \small (b) Tropical $L^p$ balls for ${p\in\{1/4,1,2,\infty\}}$ & \small (c) More complicated $\triangle$-star-convex set
    \end{tabular}
    
        \caption{$\triangle$-star-convex sets}
        \label{fig:triangStarConvSets}
    \end{figure}

    One can find examples of $\triangle$-star-convex sets in Figure~\ref{fig:triangStarConvSets}.
    Picture~(a) shows a $\min$-tropical hyperplane $H^{\min}_v$ which is $\triangle$-star-convex with kernel $v$---the apex.

    Picture~(b) displays the unit balls for tropical $L^p$ norms, which will be defined in Example~\ref{eg:tropicalLp}.
    They are nested increasingly with respect to $p$; the outer one corresponds to the tropical $L^\infty$ norm and is the only one that is not strictly $\triangle$-star-convex.
    One can recognize the triangle as the unit ball for the asymmetric tropical distance $d_\triangle$.
    The $\min$-tropical hyperplane with apex at the origin (the kernel of the $\triangle$-star-convex sets) is dotted.

    Picture~(c) shows a more complicated $\triangle$-star-convex sets.
    This case is not pure dimensional, the tropically exposed points do not form a closed set.
    Moreover, it is neither convex in the ordinary sense, nor strictly $\triangle$-star-convex.
\end{example}

\begin{proposition}
    Let $K$ be a $\triangle$-star-convex set with kernel $v$ such that $K\neq\{v\}$.
    Then $K$ is strictly $\triangle$-star-convex if and only if $K$ is strictly $\min$-tropically convex and $v$ is an interior point of $K$.
\end{proposition}

\begin{proof}
    Firstly, assume that $K$ is strictly $\triangle$-star-convex.
    For every $a,b\in K$ the $\min$-tropical segment $[a,b]_{\min}$ is a subset of $[v,a]_\triangle\cup[v,b]_\triangle$.
    Therefore, all of the points of $[a,b]_{\min}$ with the exception of $a$ and $b$ must be in the interior of $K$.
    Hence, $K$ is strictly $\min$-tropically convex.
    The fact that $v$ is an interior point is clear from the definition and our assumption that $K\neq\{v\}$.

    Conversely, assume that $K$ is strictly $\min$-tropically convex and $v$ is an interior point of $K$.
    We consider $w\in K\setminus\{v\}$ and we show that all points of $[v,w]_\triangle\setminus\{w\}$ are in the interior of $K$.
    The result is clear for non-exposed points of $[v,w]_\triangle$ as we assumed $K$ is strictly $\min$-tropically convex.
    Hence, let $u$ be an exposed point of $[v,w]_\triangle$ distinct from $w$.
    According to the discussion from Remark~\ref{rmk:geoSegment}, $u=w-(w_j-w_i)e_j$ where $i\in\argmin_k w_k$ and $j\notin\argmin_k w_k$.
    Since $(u+w)/2$ belongs to the interior of the tropical segment $[u,w]_{\min}$ and $K$ is strictly $\min$-tropically convex, then $(u+w)/2$ is an interior point of~$K$.
    Thus, for small $\delta>0$, the point $c=(u+w)/2-\delta e_i$ belongs to $K$.
    
    However, $u\in[v,c]_\triangle=S_i^{\min}\cap(c+S_i^{\max})$ as $c-u=(w-u)/2-\delta e_i=({w_j-w_i})e_j/2-\delta e_i$.
    But $u$ cannot be an exposed point of $[v,c]_\triangle$ as $c-u$ is not parallel to a vector $e_k$ for $k\in[n]$ unless $n=2$. 
    Consequently, $u$ must be an interior point of $K$ from the strict $\min$-tropical convexity of $K$, when $n\geq 3$.
    
    For the case $n=2$, we could have noticed that the exposed points of $[v,w]_\triangle$ are $v$ and~$w$, so $u$ can only be equal to $v$.
    But $v$ was already assumed to be interior.
\end{proof}

\begin{remark}
    The proof above shows that the assumption that $v$ is an interior point of $K$ is superfluous for the converse when $n\geq 3$.
\end{remark}

\begin{lemma} \label{lem:extStarConv}
    If $K$ is strictly $\triangle$-star-convex with kernel $v$, then any exposed point of $K$ from $v+S_i^{\min}$ is $i$-exposed.
\end{lemma}

\begin{proof}
    If $a\in v+S_i^{\min}$ and it is not $i$-exposed, then there exists $b\in(a+S_i^{\min})\cap K$ with $b\neq a$.
    In particular, $a\in[v,b]_\triangle\setminus\{b\}$.
    But the strict $\triangle$-star-convexity of $K$ implies that $a$ must be an interior point.
\end{proof}




\section{Tropically quasiconvex functions} \label{sec:tropQuasiconv}


A function $f:\RR^n\to\RR$ whose sub-level sets $L_{\leq\alpha}(f):=\{x:f(x)\leq\alpha\}$ are convex is called \emph{quasiconvex}.
This is a purely geometric definition, but some other sources define them as functions satisfying $f(\lambda x+({1-\lambda})y)\leq\max\{f(x),f(y)\}$ for every $x,y\in\RR^n$ and $\lambda\in[0,1]$.
The latter can be more convenient in checking quasiconvexity.
See \cite[Chapter~3]{genConvex} for more details.

We will be interested in specific \emph{tropically quasiconvex} functions.
Before we introduce them, we need some notation.
For a function $\gamma:\RR^n_{\geq 0}\to\RR$ we associate the function $\widehat\gamma:\torus{n}\to\RR$ defined by $\widehat\gamma(x)=\gamma(\widehat x)$.
We recall that $\widehat x=x-\left(\min_i x_i\right)\1$ are the canonical coordinates of $x$.

\begin{definition} \label{def:tropqconvfunc}
    We call a function $f:\torus{n}\to\RR$ \emph{$\triangle$-star-quasiconvex} with kernel~$v$ if $f(x)=\widehat\gamma(x-v)$ for some increasing function $\gamma:\RR^n_{\geq 0}\to\RR$.
    Moreover, if $\gamma$ is strictly increasing, we call $f$ \emph{strictly $\triangle$-star-quasiconvex}.
\end{definition}

We will give a geometric interpretation of $\triangle$-star-quasiconvex in Theorem~\ref{th:starQuasiConv}.
However, we prefer the definition above because it easier to check in practice.

\begin{example} \label{eg:tropicalLp}
    Considering $\gamma$ a monotonic norm \cite{monotoniNorms}, $f$ measures the distance to the kernel.
    If $v=\0$, then $f$ is a \emph{gauge} which are commonly used in convex analysis~\cite{Rockafellar:70} and location science~\cite{Nicke+Puerto:2005}.
    Gauges are sometimes dubbed ``asymmetric norms'' as they satisfy all the properties of a norm with the exception that $f(x)$ need not be equal to $f(-x)$.
    
    A famous class of monotonic norms are the $L^p$ norms.
    They give rise to $\triangle$-star-quasiconvex gauges whose expression is
    \[\gamma_p(x)=\left\{
    \begin{array}{ll}
          \sqrt[p]{\sum_{i\in[n]}\left(x_i-\min_{j\in[n]} x_j\right)^p} & \text{if }p\in[1,\infty)\\
          \max_{i\in[n]}x_i-\min_{j\in[n]}x_j & \text{if }p = \infty \\
    \end{array} 
    \right. .\]
    We call them \emph{tropical $L^p$ norms}.
    They appeared in the work of Luo~\cite{Luo:2018} under the name ``$B^p$-pseudonorms''.
    

    One can recognize the tropical $L^\infty$ norm as the tropical norm defined in \cite[\textsection 5]{Hampe:2015}.
    The relationship to the $L^\infty$ norm is stressed in \cite[Lemma~5.2.1]{Hampe:2015}.
    The tropical $L^1$ norm gives rise to the asymmetric tropical distance $d_\triangle$; this relationship is implicit in \cite[\textsection 6]{TropMedian}.
\end{example}

\begin{remark}
    The function $\widehat\gamma$ depends only on the values on $\partial\RR^n_{\geq 0}$, so we could have considered only $\partial\RR^n_{\geq 0}$ as the domain of $\gamma$.
    However, this does not increase the generality since every (strictly) increasing function defined on $\partial\RR^n_{\geq 0}$ can be extended to a (strictly) increasing function on $\RR^n_{\geq 0}$, according to the following lemma.
\end{remark}

\begin{lemma} \label{lem:extension}
    Every (strictly) increasing function $\gamma:\partial\RR^n_{\geq 0}\to\RR$ can be extended to a (strictly) increasing function $\tilde{\gamma}:\RR^n_{\geq 0}\to\RR$.
    Moreover, if $\gamma$ is continuous, then the extension can also be made continuous. 
\end{lemma}

\begin{proof}
    Consider $\tilde{\gamma}(x)=\max_{i\in[n]}\gamma(x_{-i},0_i)+\prod_{i\in[n]}x_i$.
    Clearly, this is continuous if $\gamma$ is, as being a composition of continuous functions.
    Moreover, $\tilde\gamma(x)=\gamma(x)$ for every $\vec x\in\partial\RR^n_{\geq 0}$, due to monotonicity of $\gamma$ and the fact that $x_1x_2\dots x_n=0$ for $x\in\partial\RR^n_{\geq 0}$.

    If $ x\leq y$, then $x_{-i}\leq y_{-i}$ for all $i\in[n]$, where $x_{-i}$ is obtained from $x$ by removing the $i$th entry.
    Therefore, $\gamma(x_{-i},0_i)\leq\gamma(y_{-i},0_i)$ for every $i\in[n]$, which implies $\tilde\gamma(x)\leq\tilde\gamma(y)$ after using $\prod_j x_j\leq\prod_j y_j$.
    In other words, $\tilde\gamma$ is increasing.

    Moreover, if $\gamma$ is strictly increasing and $x\neq y$ we have two cases.
    On the one hand, if $y\in\partial\RR^n_{\geq 0}$, then $x\in\partial\RR^n_{\geq 0}$ so $\tilde\gamma(x)=\gamma(x)<\gamma(y)=\tilde\gamma(y)$.
    
    On the other hand, if $y\in\RR^n_{>0}$, then $\prod_j x_j<\prod_j y_j$.
    Using the last inequality with $\max_{i\in[n]}\gamma(x_{-i},0_i)\leq\max_{i\in[n]}\gamma(y_{-i},0_i)$, we obtain $\tilde\gamma(x)<\tilde\gamma(y)$.
    Accordingly, $\tilde\gamma$ is strictly increasing if $\gamma$ is strictly increasing.
\end{proof}

The following result explains why the functions from Definition~\ref{def:tropqconvfunc} deserve the name ``$\triangle$-star-quasiconvex''.

\begin{theorem} \label{th:starQuasiConv}
    Let $f:\torus{n}\to\RR$ be a continuous function.
    Then $f$ is (strictly) $\triangle$-star-quasiconvex if and only if all of its non-empty sub-level sets are (strictly) $\triangle$-star convex with the same kernel.
\end{theorem}

\begin{proof}
    After an eventual translation, we can assume that the kernel is $\0$.

    Firstly, assume $f$ is $\triangle$-star-quasiconvex and let $\alpha\in\RR^n$ arbitrary such that $L_{\leq\alpha}(f)$ is non-empty.
    Let $\gamma:\RR^n\to\RR$ increasing such that $f(x)=\widehat\gamma(x)$.
    
    Let $w\in L_{\leq\alpha}(f)$ and choose $i\in[n]$ such that $w\in S_i^{\min}$.
    Since $\gamma$ is increasing, the points $x\in\RR^n$ satisfying $\0\leq x\leq\widehat w$ belong to $L_{\leq\alpha}(\gamma)$.
    This set projects onto $[\0,w]_\triangle$ showing that $[\0,w]_\triangle\subseteq L_{\leq\alpha}(f)$.
    Since $w$ was selected arbitrarily, $L_{\leq\alpha}(f)$ must be $\triangle$-star convex with kernel $\0$.

    If $f$ is strictly $\triangle$-star-quasiconvex, then the points satisfying $\0\leq x\leq \widehat w$ different from $\widehat w$ actually belong to $L_{<\alpha}(f)$.
    Due to the continuity of $f$, this coincides with the interior of $L_{\leq\alpha}(f)$.
    This shows that $L_{\leq\alpha}(f)$ is strictly $\triangle$-star-convex.

    Conversely, assume that $L_{\leq\alpha}(f)$ is $\triangle$-star-convex with kernel $\0$ for every $\alpha\geq f(\0)$.
    Take $\gamma:\partial\RR^n_{\geq 0}\to\RR$ defined as $\gamma(x)=f(\widehat x)$ for $x\in\partial\RR^n_{\geq 0}$.
    Using Lemma~\ref{lem:extension} it is enough to show that $\gamma$ is increasing.

    Let $x$ and $y$ arbitrary points of $\partial\RR^n_{\geq 0}$ such that $x\leq y$.
    The order segment $[\0,y]_{\leq}$ projects onto $[\0,y]_{\triangle}$ which belongs to $L_{\leq f(y)}(f)$.
    Due to the $\triangle$-star-convexity of sub-level sets, we obtain $\gamma(x)=f(x)\leq f(y)=\gamma(y)$.

    If we have strict $\triangle$-star-convexity, then $[0,y]_\triangle\setminus\{y\}$ is contained in the interior of $L_{\leq f(y)}(f)$ which coincides to $L_{<f(y)}(f)$.
    Hence, we obtain $\gamma(x)<\gamma(y)$ for this case. 
\end{proof}

\begin{remark}
    The continuity of $f$ is relevant only for strictly $\triangle$-star-quasiconvex functions.
    Without continuity, only the strict $\triangle$-star-convexity of the sub-level sets is not sufficient for $f$ to be strictly $\triangle$-star-quasiconvex.
    This is similar to the case of ordinary quasiconvex functions; cf. \cite[Proposition~3.28]{genConvex} and \cite[Example~3.3]{genConvex}.
\end{remark}

We will see that convexity, in the ordinary sense, will also be helpful for our applications.
We give a simple criterion for checking when a $\triangle$-star-quasiconvex function is convex.

\begin{lemma} \label{lem:convexGammaHat}
    If $\gamma$ is increasing and (strictly) convex, then $\widehat\gamma$ is (strictly) convex.
\end{lemma}

\begin{proof}
    Let $x,y\in\RR^n$ and $\lambda\in[0,1]$.
    
    We have $\min_j\left(\lambda x_i+(1-\lambda) y_i\right)\geq\lambda\min_i x_i+(1-\lambda)\min_i y_i$ as $\lambda,1-\lambda\geq 0$.
    Hence, $\lambda x+(1-\lambda) y-\min_j\left(\lambda x_i+(1-\lambda) y_i\right)\1\leq\lambda( x-\min_i x_i\1)+(1-\lambda)( y-\min_i y_i\1)$.
    Since $\gamma$ is convex and increasing, we obtain
    \begin{equation} \label{eq:ineqGammaHat}
    \begin{split}
        \widehat\gamma\left(\lambda x+(1-\lambda) y\right) & \leq\gamma\left(\lambda( x-\min_i x_i\1)+(1-\lambda)(y-\min_i y_i\1)\right) \\
        & \leq\lambda\,\gamma(x-\min_i x_i\1)+(1-\lambda)\,\gamma(y-\min_i y_i\1) \\
        &=\lambda\widehat\gamma(x)+(1-\lambda)\widehat\gamma(y).  
    \end{split}
    \end{equation}

    If $\gamma$ is strictly convex and $x\neq y$ modulo $\RR\1$, then the second inequality from~\eqref{eq:ineqGammaHat} is strict, so $\widehat\gamma\left(\lambda x+(1-\lambda) y\right)<\lambda\widehat\gamma(x)+(1-\lambda)\widehat\gamma(y)$.
    Thus, $\widehat\gamma$ is strictly convex if $\gamma$ is strictly convex.
\end{proof}

\section{Tropically convex location problems} \label{sec:locProb}

We will consider some input points $v_1,\dots,v_m$ in $\torus{n}$.
We measure the distance (or dissimilarity) from $x\in\torus{n}$ to a point $v_i$ using a $\triangle$-star-quasiconvex function $f_i$ having kernel $v_i$.
We consider increasing functions $\gamma_i:\RR^n\to\RR$ such that $f_i(x)=\widehat\gamma_i(x-v_i)$.
Without loss of generality, we assume $\gamma_i(\0)=0$, so that all dissimilarities are non-negative.

The purpose of location problems is to find a point as close (or similar) as possible to the input points, depending on some criterion; usually, the optimal location is a minimum of an objective function $h:\torus{n}\to\RR$.
The function $h$ is constructed using an increasing function $g:\RR^m_{\geq 0}\to\RR$, which aggregates the distances to the input points.
Formally, we define $h(x)=g\left(f_1(x),\dots,f_m(x)\right)$. 

Since $f_i$ measures the distance or dissimilarity from $x$ to $v_i$ and $g$ is increasing, the minima of $h$ record a global closeness to the input points.
In most studied location problems, we would have a distance $d$ on $\torus{n}$ and set $f_i(x)=d(x,v_i)$.
Common choices of $g$ are $g(x)=x_1+\dots+x_m$, for the median or Fermat--Weber problem, $g(x)=\max_{i\in[m]}x_i$ for the center problem~\cite{LocationScience:2019}, or $g(x)=x_1^2+\dots+x_m^2$, for defining the Fr\'{e}chet mean~\cite{frechet:1948}.
Nevertheless, we will allow $g$ to be an arbitrary increasing function.
We will assume that $h$ has a minimum, which happens, e.g., when $h$ is lower semi-continuous.

\begin{theorem} \label{th:locationTropConv}
    Let $h$ be as above.
    Then there is a minimum of $h$ belonging to $\tconv^{\max}(v_1,\dots, v_m)$.
    Moreover, if $g$ is strictly increasing and at least one of $f_1,\dots,f_m$ is strictly $\triangle$-star-quasiconvex, then all the minima of $h$ are contained in $\tconv^{\max}(v_1,\dots,v_m)$.
\end{theorem}

\begin{proof}
    Consider $x\notin\tconv^{\max}(v_1,\dots,v_m)$ which is a minimum of $h$.
    Thus there exists $k\in[n]$ such that $k\notin\argmin_j(x_j-v_{ij})$ for all $i\in[m]$.
    Set $\delta_i:=x_k-v_{ik}-\min_j(x_j-v_{ij})$ for all $i$, and $\delta=\min_i\delta_i$, which is strictly positive by the consideration of $k$.
    
    Note that $f_i(x-\delta e_k)=\gamma_i\left(x-v_i-\delta e_k-\min_j({x_j-v_{ij})\1}\right)\leq\gamma_i(x-v_i-{\min_j(x_j-v_{ij})\1})=f_i(x)$ for all $i\in[n]$.
    Hence $h(x-\delta e_k)\leq h(x)$.

    Note that the inequality above is strict if $g$ and some $\gamma_\ell$ are strictly increasing.
    Indeed, in that case, we must have $f_\ell(x-\delta e_k)<f_\ell(x)$, so we use the strict increase of $g$ in the $\ell$th entry.
    That would contradict the optimality of~$x$, so the second statement of the theorem holds.

    For the first statement, we can only infer that $x-\delta e_k$ is also a minimum of $h$.
    Hence, we can find an optimum of $h$ in $\tconv^{\max}(v_1,\dots,v_m)$ by moving $x$ in directions $-e_k$ for indices $k$ as above.

    To be more precise, we collect in $D(x)$ the possible elementary descent directions from $x$; formally $D(x):=\bigcap_{i\in[m]}\left([n]\setminus\argmin_{j\in[n]}(x_j-v_{ij})\right)$.
    Notice that $k\in D(x)$, but $k\notin D(x-\delta e_k)$.
    Moreover, $D(x-\delta e_k)\subsetneq D(x)$, as the $\argmin$ functions only increase by our move in a descent direction.
    Thus, replacing $x$ by $x-\delta e_k$, we find a minimum with smaller $D(x)$.
    We can repeat the procedure to construct a minimum $x^\star$ of $h$ with $D(x^\star)=\emptyset$.
    The last condition is equivalent to $x^\star\in\tconv^{\max}(v_1,\dots,v_m)$ due to \cite[Proposition~5.37]{ETC}.
\end{proof}

\begin{remark}
    The regions of $f_i$ where it looks like a monotonic function are induced by the $\min$-tropical hyperplane based at $v_i$.
    Those hyperplanes defined the $\max$-tropical polytope generated by the input points, explaining why we look at the $\max$-tropical convex hull, instead of the $\min$ analogue.
\end{remark}

The following lemma presents cases when there is a unique optimum location.
We recall that a gauge $\gamma$ is called \emph{strictly convex} if $\gamma(\lambda x+(1-\lambda)y)<1$ for every $\lambda\in(0,1)$ and $x,y\in\torus{n}$ with $\gamma(x)=\gamma(y)=1$, although they are not strictly convex functions.


\begin{lemma} \label{lem:unique-sol}
Assume that $g,f_1,\dots,f_m$ are convex, $g$ is strictly increasing, and at least one of the following conditions holds:
\begin{itemize}
    \item[a)] at least one $f_i$ is strictly convex; or
    \item[b)] all $f_i$ are strictly convex gauges and the points $v_1,\dots, v_m$ are not collinear.
\end{itemize}
Then $h$ is strictly convex.
In particular, it has a unique minimum.
\end{lemma}

\begin{proof}
Consider arbitrary distinct points $x, y\in\RR^n/\RR\1$ and a scalar $\lambda\in(0,1)$.

For case a), we have $f_i(\lambda x+(1-\lambda) y-v_i)<\lambda f_i(x-v_i)+(1-\lambda)f_i(x-v_i)$.
Since $g$ is convex and strictly increasing and the functions $f_j$ convex, we obtain
$h(\lambda x+(1-\lambda) y)<\lambda h(x)+(1-\lambda)h(y)$.
So $h$ must be strictly convex.

For case b), at least one of the points $v_i$ is not on the line through $x$ and $y$.
Then $x-v_i$ and $y-v_i$ they are not parallel and the strict convexity of the unit ball defined by $f_i$ implies that $f_i(\lambda x+(1-\lambda) y-v_i)<\lambda f_i(x-v_i)+(1-\lambda)f_i(x-v_i)$.
The rest of the proof is identical to case a).
\end{proof}

\subsection{Examples} \label{subsec:examples}

Here we review the tropical location problems from literature that fall in our category, i.e. an optimum belongs to the tropical convex hull of the input.

%



\begin{example}[Tropical Fermat--Weber and Fr\'{e}chet problems] \label{eg:tropFW+Frechet}
    To the best of our knowledge, the first one-point location problems in tropical geometry are proposed by Lin et~al.~\cite{LSTY:2017}.
    They suggest the study of Fermat--Weber points and Fr\'{e}chet means under the symmetric tropical distance $d_{\trop}$.
    The goal was to relate them to tropical convexity for applications in phylogenetics.

    However, they noticed that tropical Fermat--Weber points might lie outside the tropical convex hull of the input points leading to medians that cannot be interpreted easily in biological applications \cite[Example~27]{LSTY:2017}.
    However, Theorem~\ref{th:locationTropConv} says that it is possible to find an optimum in the tropical convex hull.
    This was already noticed for the tropical Fermat--Weber points \cite[Lemma~3.5]{Page+Yoshida+Zhang:2020} but it was unknown, until now, for tropical Fr\'{e}chet means.
\end{example}

\begin{example}[Tropical center] \label{eg:tropCenter}
    Consider the case $f_i(x)=d_\triangle(v_i,x)$ and $g(y)=\max(y_1,\dots,y_m)$.
    This can be interpreted as the center of the minimum $\max$-tropical $L^1$ ball enclosing the points $v_1,\dots,v_m$.
    The tropical center appears in \cite[Example~23]{TropMedian}, but the details are omitted.
    
    If we choose representatives of the input points in $\cH=\{x\in\RR^n:{x_1+\dots+x_n=0}\}$, the optimum can be obtained by solving the linear program:
\begin{equation} \label{lp:trop_center}
    \begin{array}{ll@{}ll}
      \text{minimize}  & \displaystyle  n\cdot t & \\
      \text{subject to}& \displaystyle v_{ij} - x_j \ \leq \ t \,, & \quad \text{for } i\in[m] \text{ and } j\in[n] \\
                       & \displaystyle x_{1} + \dots + x_{n} \ = \ 0  &
    \end{array}.
  \end{equation}

    Note that the $x$-coordinates of the optimal solutions are equal, modulo $\RR\1$, to the $x$-coordinates of the linear program
\begin{equation} \label{lp:equi_trop_center}
    \begin{array}{ll@{}ll}
      \text{minimize}  & \displaystyle  n\cdot t+\sum_{j=1}^n x_j & \\
      \text{subject to}& \displaystyle v_{ij} - x_j \ \leq \ t \,, & \quad \text{for } i\in[m] \text{ and } j\in[n] 
    \end{array}.
  \end{equation}

    Let $(t^\star,x^\star)$ an optimal solution of \eqref{lp:equi_trop_center}.
    For any solution of \eqref{lp:equi_trop_center} we have $t+x_j\geq\max_{i\in[m]}v_{ij}=:V_j$.
    In particular, $x^\star$ will have the smallest entries if we actually have equality: $t^\star+x^\star_j=V_j$, otherwise we can replace $x^\star$ by some $x^\star-\varepsilon e_i$ to minimize the objective function.
    This implies $x^\star = V$ modulo $\RR\1$; in particular, the solution is unique in $\torus{n}$.

    Even if we do not have $g$ strictly increasing, the uniqueness and Theorem~\ref{th:locationTropConv} ensures that the optimum is in the tropical convex hull.
    However, this could have been noticed from the closed form $V=\bigvee_i v_i$ for $v_1,\dots,v_m\in\cH$.
\end{example}

\begin{example}[Transportation problems] \label{eg:transport}
    Consider $\lambda_1,\dots,\lambda_n>0$ and $\triangle(\lambda)$ the simplex in $\torus{n}$ whose vertices are $e_i/\lambda_i$.
    Then $\gamma_{\triangle(\lambda)}(x)=\sum_{i}\lambda_i x_i-\left(\sum_i\lambda_i\right)\min_j x_j$
    is the gauge on $\torus{n}$ whose unit ball is $\triangle(\lambda)$.
    
    The (weighted) Fermat--Weber problem $\sum_{i\in[m]}w_i\gamma_{\triangle(\lambda)}(x-v_i)$ is equivalent to a transportation problem and every transportation problem can be reduced to this case; to see this better, write it as a linear program after scaling the weights $w_i$ such that $\sum_i w_i=\sum_j \lambda_j$ (this change does not influence the optimum).
    This was firstly noticed in~\cite{TropMedian}, where the authors focused on the case $\lambda_1=\dots=\lambda_n$.
    The corresponding optimum is called a \emph{tropical median} in the work cited.

    The optimal point is called a $\lambda$-splitter by Tokuyama and Nakano~\cite{Tokuyama+Nakano:1995}, but no metric interpretation was mentioned. 
    The authors gave a condition of partitioning the space in $n$ region in an equal fashion with some weights coming from $\lambda$ and $w$; this can be seen as a reinterpretation of the first-order optimality condition for the corresponding Fermat--Weber problem. 
    As a $\lambda$-splitter, it appeared in statistics~\cite{Gallegos+Ritter:2010} and as a particular case of Minkowski partition problems~\cite{AHA:1998}.
\end{example}

\begin{example}[Locating tropical hyperplanes]
    The tropical hyperplanes are parametrized by $\RR^n/\RR\1$ by their identification with their apex.
    Moreover, we have $d_{\trop}(a,H_x^{\max})=(x-a)_{(2)}-(x-a)_{(1)}$.
    For a vector $y$, we denote by $y_{(k)}$ the $k$th smallest entry, also known as the $k$th order statistic.
    Note that the aforementioned distance is $\triangle$-star-quasiconvex with apex $a$; the easiest to see this is noticing that the second order statistic is increasing.
    Therefore, our general location problems cover the case of locating tropical hyperplanes.
    
    The best-fit tropical hyperplane with with $L^1$ error, i.e. $g$ is the $L^1$ norm, was considered by Yoshida, Zhang, and Zhang as part of \emph{tropical principal component analysis}~\cite{YZZ:2019}.
    
    The case of $L^\infty$ error was considered by Akian et al. \cite{tropicalLinearRegression} for applications to auction theory and called \emph{tropical linear regression}.
    They also show that the problem is polynomial-time equivalent to mean-payoff games \cite[Corollary~4.15]{tropicalLinearRegression} and, using $d_{\trop}(a,H_x^{\max})=d_{\trop}(x,H_a^{\min})$, that it is dual to the problem of finding the largest inscribed ball in the tropical convex hull of the input points \cite[Theorem~4.6]{tropicalLinearRegression}.  
\end{example}

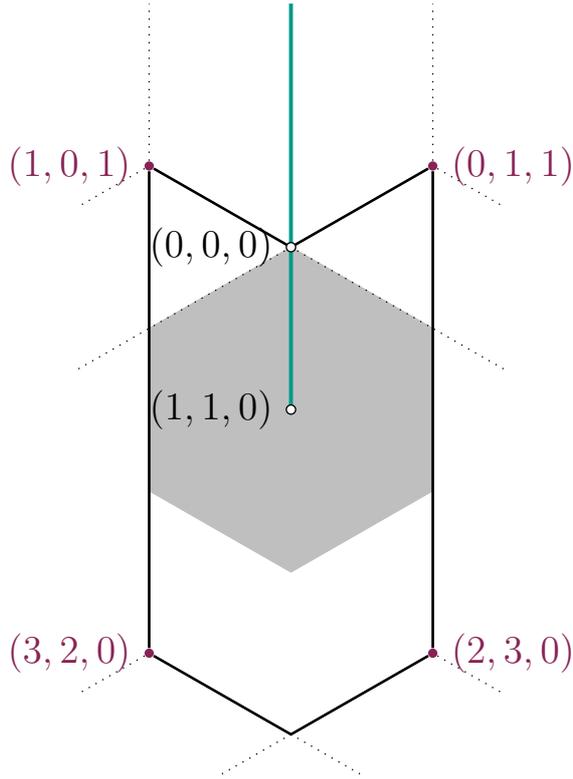
\begin{figure}
    \centering
         \begin{tikzpicture}[scale=1.75,x={(-0.866cm,-0.5cm)},z={(0cm,1cm)},y={(0.866cm,-0.5cm)}]
        \filldraw[gray!50!white] (2,1,0) -- (1,0,0) -- (0,0,0) -- (0,1,0) -- (1,2,0) -- (2,2,0) -- cycle;

        \node[circle,fill=tolpurple,inner sep=1pt, label={[tolpurple]right:$(0,1,1)$}] (a) at (-1,0,0) {};
        \node[circle,fill=tolpurple,inner sep=1pt, label={[tolpurple]left:$(1,0,1)$}] (b) at (0,-1,0) {};
        \node[circle,fill=tolpurple,inner sep=1pt, label={[tolpurple]left:$(3,2,0)$}] (c) at (1,0,-2) {};
        \node[circle,fill=tolpurple,inner sep=1pt, label={[tolpurple]right:$(2,3,0)$}] (d) at (0,1,-2) {};

        \draw[dotted] (a) -- (1.5,0,0);
        \draw[dotted] (a) -- (-1,0.5,0);
        \draw[dotted] (a) -- (-1,0,1);
        
        \draw[dotted] (b) -- (0.5,-1,0);
        \draw[dotted] (b) -- (0,1.5,0);
        \draw[dotted] (b) -- (0,-1,1);
        
        \draw[dotted] (c) -- (1.5,0,-2);
        \draw[dotted] (c) -- (1,1.5,-2);
        
        \draw[dotted] (d) -- (1.5,1,-2);
        \draw[dotted] (d) -- (0,1.5,-2);

        \draw[thick] (a) -- (0,0,0) --
        (b) -- (c) -- (0,0,-3) -- (d) -- (a);
        \draw[very thick,tolgreen] (0,0,-1) -- (0,0,1.5);

        \node[circle,draw=black,fill=white,inner sep=1pt, label={[black]left:$(1,1,0)$}] (tc) at (1,1,0) {};
        \node[circle,draw=black,fill=white,inner sep=1pt, label={[black]left:$(0,0,0)$}] (O) at (0,0,0) {};

    \end{tikzpicture}
    \caption{Input points (purple) with their convex hull (black boundary) and various locations from the examples of \textsection\ref{subsec:examples}; see the discussion after the examples}
    \label{fig:trop-locations}
\end{figure}

To end this subsection, we compute the optimal location from the examples above for specific input points.
We consider the points from \cite[\textsection 4]{tropicalLinearRegression} which are given by the columns of the matrix
\[
    V \ = \ \begin{pmatrix}
      0 & 1 & 3 & 2 \\
      1 & 0 & 2 & 3 \\
      1 & 1 & 0 & 0 \\
    \end{pmatrix}.
\]

For this input, there is a unique tropical Fr\'{e}chet point, $(1,1,0)$, but the set of tropical Fermat--Weber points is a hexagon, marked with grey in Figure~\ref{fig:trop-locations}.
We remark that $V$ has two axes of symmetry and $(1,1,0)$ is their intersection.

The point $(1,1,0)$ is also the tropical center of $V$, while the tropical median is $(0,0,0)$.
The latter point is the also the unique apex of the best-fit tropical hyperplane with $L^1$ error of \cite{YZZ:2019}.
It is also a solution of the tropical linear regression, but not the unique one.
The apices of the best-fit tropical hyperplanes with $L^\infty$ error are of the form $(\lambda,\lambda,0)$ with $\lambda\leq 1$ and their set is pictured with green in Figure~\ref{fig:trop-locations}.


\subsection{Regularization} \label{subsec:regular}
    In some cases, we cannot expend $g$ to be strictly increasing or all the dissimilarity functions $f_i$ to be strictly $\triangle$-star-quasiconvex.
    Hence, a minimization algorithm might return a point outside the $\max$-tropical convex hull of the input points, when there are multiple solutions.
    In this subsection, we show how we could try to arrive to a solution belonging to $\tconv^{\max}(v_1,\dots,v_m)$ through a regularized formulation.

    The idea of regularization is to consider a small parameter $\lambda>0$ and a nicely behaved function $f_{m+1}:\torus{n}\to\RR_{\geq 0}$ and try to solve the optimization problem
    \[\text{minimize }g\left(f_1(x),\dots,f_m(x)\right)+\lambda f_{m+1}(x).\]
    For our purposes, $f_{m+1}$ is nicely behaved if it is strictly $\triangle$-star-quasiconvex with a kernel from $\tconv^{\max}(v_1,\dots,v_m)$.
    An easy choice for $v$ is the tropical center from Example~\ref{eg:tropCenter}.

    This is also a location problem with $g_\lambda:\RR^{m+1}_{\geq 0}\to\RR$ given by $g_\lambda(x_1,\dots,x_m,x_{m+1})=g(x_1,\dots,x_m)+\lambda x_{m+1}$ and the optimality criterion is the function $h_\lambda:\torus{n}\to\RR$ given by $h_\lambda(x)=g_\lambda(f_1(x),\dots,f_{m+1}(x))$
    Note that $g_\lambda$ is strictly increasing in the $(m+1)$-st entry for every $\lambda>0$.

    Checking more carefully the proof of Theorem~\ref{th:locationTropConv}, the second statement holds     if $f_\ell$ is strictly $\triangle$-star-quasiconvex and $g$ strictly increasing in its $\ell$-th entry.
    We use this property for the regularization.
    Therefore, we obtain the following direct consequence of Theorem~\ref{th:locationTropConv}.

\begin{corollary} \label{cor:reg}
    For every $\lambda>0$, all the minima of $h_\lambda$ lie in $\tconv^{\max}(v_1,\dots,v_m)$.
\end{corollary}

The influence of the term $f_{m+1}$ decreases as $\lambda$ goes to $0$.
If the functions are regular enough, we expect that a collection of optima $x^\star_\lambda$ of $h_\lambda$ to converge to an optimum of $h$.
In fact, $x^\star_\lambda$ will be an optimum of $h$ for $\lambda$ sufficiently small if $h$ is polyhedral convex and $f_{m+1}$ is Lipschitz continuous.

\begin{proposition} \label{prop:exactRegSol}
    If $h$ is polyhedral convex and $f_{m+1}$ is a convex function with sub-linear growth, then there exists $\lambda_0>0$ such that all minima of $h_\lambda$ are also minima of $h$ for every $\lambda<\lambda_0$.
\end{proposition}

The proof is quite technical using the differential theory from convex analysis so it is given in the appendix.
We stress that Proposition~\ref{prop:exactRegSol} can be useful for studying the tropical Fermat--Weber problem from \cite{Lin+Yoshida:2018}.
Without regularization, it has undesirable behaviour for applications to biology; cf.~\cite[\textsection 5.2]{TropMedian}.




\section{Location problems with tropically convex sites} \label{sec:locProbTropConvSites}

Location problems can appear also when facilities are regions of the ambient space and not only points.
Here, we consider such a generalization where the sites are tropically convex sets.

In the previous section, we used different distances to the input points.
Here, we will measure our dissimilarities in a uniform way, by fixing an increasing function $\gamma:\RR^n_{\geq 0}\to\RR$ and considering $d_\gamma(x,y)=\widehat\gamma(y-x)$.
We than say that $d_\gamma$ is $\triangle$-star-quasiconvex; if $\gamma$ is strictly increasing we say that $d_\gamma$ is strictly $\triangle$-star-quasiconvex.
This allows a clear definition of a distance from a region to a point: $d_\gamma(A,x):=\inf_{y\in A}d_\gamma(y,x)$.




For a closed $\max$-tropical cone $K\subseteq\RR^n$ we define the projection $\pi_K:\RR^n\to K$ as $\pi_K(x)=\max\{y\in K:y\leq x\}$.
We note that $\pi_K(x+\lambda\1)=\pi_K(x)+\lambda\1$ for every $x\in\RR^n$ and $\lambda\in\RR$, so it induces a well-defined function $\pi_{K/\RR\1}:\torus{n}\to K/\RR\1$ called the \emph{tropical projection} onto the $\max$-tropically convex set $K/\RR\1$.


The following lemma gives an explicit formula for the tropical projection and it characterizes it as a closest point under $d_\gamma$.
We omit the proof, as it is a classical result, shown when $\gamma$ is the maximum norm in \cite[\textsection 3]{CGQ:2004} and for a general tropical $L^p$ norm in \cite[Theorem~4.6]{Luo:2018}.

\begin{lemma} \label{lem:tropProj}
Let $A$ be a closed $\max$-tropically convex set.
Then the tropical projection $\pi_A(x)$ of a point $x$ has the entries
\begin{equation} \label{eq:tropProj}
    \pi_A(x)_i=\max_{a\in A}\left(a_i+\min_{j\in[n]}(x_j-a_j)\right).  
\end{equation}

Moreover, $d_\gamma(A, x)=d_\gamma(\pi_A(x), x)$ and $\pi_A(x)$ is the unique point whose distance to $x$ equals $d_\gamma(A,x)$ if $d_\gamma$ is strictly $\triangle$-star-quasiconvex.
\end{lemma}

\begin{remark}
    In fact, the maximum expression of the tropical projection from Lemma~\ref{lem:tropProj} can be taken over the extremal points, in the case of tropical polytopes \cite[Propositon~5.24]{ETC}.
    A similar result seems similar for general convex sets, but the form above is sufficient for our purposes.
\end{remark}

From now on, our given sites are closed $\max$-tropically convex sites $A_1,\dots,A_m$ in $\torus{n}$.
Similar to section~\ref{sec:locProb}, the objective function is $h=g\left(d_{\gamma}(A_1,x),\dots,d_{\gamma}(A_m,x)\right)$, where $g:\RR^m_{\geq 0}\to\RR_{\geq 0}$ is increasing.

\begin{theorem} \label{th:locationSetSites}
There exists an minimum of $h$ lying in the tropical convex hull of the input $\tconv^{\max}(A_1\cup\dots\cup A_m)$.
Moreover, if $g$ and $\gamma$ are strictly increasing, then all the minima of $h$ lie in $\tconv^{\max}(A_1\cup\dots\cup A_m)$.
\end{theorem}

\begin{proof}
    If $x\notin\tconv^{\max}(A_1\cup\dots\cup A_m)$, then \cite[Proposition~5.37]{ETC} entails the existence of an index $\ell\in[n]$ such that $\min_{j\neq\ell}({x_j-a_j})<{x_\ell-a_\ell}$ for every $a\in\tconv^{\max}(A_1\cup\dots\cup A_m)$. 
    Since $A_1,\dots,A_m$ are closed sets, then there exists an open ball around $x$ not intersecting the union of these sets.
    Thus, for $\delta>0$ sufficiently small and $y=x-\delta e_\ell$ we have $\min_j(y_j-a_j)=\min_j(x_j-a_j)$ for every $a\in\tconv^{\max}(A_1\cup\dots\cup A_m)$.
    Therefore, equation~\eqref{eq:tropProj} implies $\pi_{A_i}(y)=\pi_{A_i}(x)$ for all $i\in[m]$.

    Note that $y-\pi_{A_i}(y)=x-\pi_{A_i}(x)-\delta e_\ell\leq x-\pi_{A_i}(x)$.
    Since $\gamma$ is increasing, we have $d_{\gamma}(A_i,y)=\gamma(y-\pi_{A_i}(y))\leq\gamma(x-\pi_{A_i}(x))=d_{\gamma}(A_i,x)$ for every $i\in[m]$.
    Moreover, if $\gamma$ is strictly increasing we get $d_{\gamma}(A_i,y)<d_{\gamma}(A_i,x)$.
%

    In other words, going from $x$ in the direction $-e_{\ell}$ we obtain a decrease in all the distances $d_{\gamma}(A_i,x)$; in particular, a decrease of $h$.
    Using this observation, the rest of the proof is identical to the proof of Theorem~\ref{th:locationTropConv}.
\end{proof}

\section{Tropically convex consensus methods} \label{sec:tropConsensus}

In this section, we focus on applications to phylogenetics---the study of evolutionary history of species \cite{Felsenstein:2003,SempleSteel:2003}.
The information is represented as an evolutionary tree, or phylogeny, which are trees whose leaves are labeled by the name of the species.
In this paper, we will deal only with trees that encode the evolution from a common ancestor and possess a molecular clock.

To be more formal, we have a finite set $\sX$ containing the names of the species and a rooted tree whose leaves are in bijection with $\sX$; the root corresponds to the most recent ancestor of all the species into consideration.
The time is represented as positive weights on the edges, which gives a way to measure distances between nodes in the trees.
What is more we assume that the distance from the root to any leaf is the same; it means that the same time is measured from the evolution of the most recent common ancestor (MRCA) of all species and any element of $\sX$.
Such trees are called \emph{equidistant}.

To a rooted phylogeny $T$ we associate a distance matrix $D\in\RR^{\sX\times\sX}$ where the entry $D_{ij}$ represents the distance between the leaves labelled $i$ and $j$ in $T$.
It is known that $T$ is equidistant if and only if $D$ is ultrametric \cite[Theorem~7.2.5]{SempleSteel:2003}, i.e.
\begin{equation} \label{eq:ultrametric}
    D_{ij}\leq\max(D_{ik},D_{kj})\quad\forall\, i,j,k\in\sX.  
\end{equation}
Hence, we will not distinguish between equidistant trees and ultrametric matrices in the rest of the paper.

Because $D$ is symmetric and has zero entries on the diagonal, we can see it as a point of $\RR^{\binom{\sX}{2}}$.
We define the \emph{tree space} $\cT_\sX$ as the image of space of all ultrametrics in $\torus{\binom{\sX}{2}}$.
Due to \cite[Proposition~3]{Ardila+Klivans:2006}, this is homeomorphic to the BHV space defined in \cite{BHV}.
We note that the ultrametric condition~\eqref{eq:ultrametric} implies that $\cT_\sX$ is $\max$-tropically convex.

We are interested in \emph{consensus methods}: given as input multiple phylogenies on $\sX$, find an evolutionary tree on $\sX$ being as similar as possible to the input trees.
This is a common problem in evolutionary biology, as multiple distinct trees arise from the statistical procedures or from the multiple methods to reconstruct phylogenies from different data; see \cite{Bryant:2003} or \cite[Chapter~30]{Felsenstein:2003} for details.

A consensus method can be seen as a location statistic in the tree space.
Since the latter is $\max$-tropically convex, there were many attempts to exploit this geometric structure to obtain relevant information \cite{TropMedian,LSTY:2017,Lin+Yoshida:2018,Page+Yoshida+Zhang:2020}.
We are interested in tropically convex consensus methods, defined in \cite{TropMedian}.

\begin{definition}
    A consensus method $c$ is \emph{tropically convex} if $c(T_1,\dots,T_m)\in\tconv^{\max}(T_1,\dots,T_m)$
    for every $m\geq 1$ and $T_1,\dots,T_m\in\cT_\sX$.
\end{definition}

The location problems discussed in the previous section give rise to tropically convex consensus methods.
Note that we do not need to impose the restriction that the optimum to lie in $\cT_\sX$.
It is automatically satisfied from the tropical convexity of $\cT_\sX$ and Theorem~\ref{th:locationTropConv}.
This observation ensured that tropical median consensus methods are fast to compute \cite[\textsection 5.3]{TropMedian}.

Tropically convex consensus methods are particularly interesting because they preserve relationships from the input trees.
To explain this more clearly, we firstly need some terminology: two subsets of taxa $A,B$ form a \emph{nesting} in $T$, and we denote it by $A<B$, if the MRCA of $A$ in $T$ is a strict descendant of the MRCA of $A\cup B$.
If $D$ is the ultrametric associated to $T$, then we can write the condition as
\begin{equation} \label{eq:nesting}
    \max_{i,j\in A}D_{ij}<\max_{k,\ell\in A\cup B}D_{k\ell}.
\end{equation}

We say that a consensus method $c$ is \emph{Pareto on nestings} if $c(T_1,\dots,T_m)$ displays the nesting $A<B$ whenever $A<B$ appears in all input trees $T_1,\dots,T_m$.
The consensus method $c$ is called \emph{co-Pareto on nestings} if $c(T_1,\dots,T_m)$ does not display the nesting $A<B$ unless $A<B$ appears in some input tree $T_i$.
These conditions are desirable for consensus methods \cite{Bryant:2003,propSupertree}.

\begin{remark} \label{rmk:geomPareto}
    It is useful to see these properties from a geometric point of view.
    Consider $\cT_\sX(A<B)$ the subset of $\cT_\sX$ consisting of trees displaying the nesting $A<B$; it is described by \eqref{eq:nesting}.
    We also make the notation $\cT_\sX(A\not <B)$ for the complement ${\cT_\sX\setminus\cT_\sX({A<B})}$, which is the set of trees not displaying $A<B$.
    
    Then $c$ is Pareto on nestings if and only if for every nesting $A<B$ and trees $T_1,\dots,T_m\in\cT_\sX(A<B)$ we have $c(T_1,\dots,T_m)\in\cT_\sX(A<B)$.
    We also note that $c$ is co-Pareto on nestings if and only if for every nesting $A<B$ and trees $T_1,\dots,T_m\in\cT_\sX(A\not <B)$ we have $c(T_1,\dots,T_m)\in\cT_\sX(A\not <B)$.
\end{remark}

The next result shows that tropically convex consensus methods have both Pareto and co-Pareto properties, being an improved version of \cite[Proposition~22]{TropMedian}.
Thus, we have a large class of consensus methods satisfying both properties.
This is remarkable, as no such consensus method is listed in the surveys \cite{Bryant:2003,propSupertree}.

\begin{proposition} \label{prop:consensusPareto}
    Tropically convex consensus methods are Pareto and co-Pareto on nestings.
\end{proposition}

\begin{proof}
    For every nesting $A<B$, the set $\cT_\sX(A<B)$ is $\max$-tropically convex as \eqref{eq:nesting} describes an open $\max$-tropical halfspace.
    Whence, Remark~\ref{rmk:geomPareto} implies that tropically convex consensus methods are Pareto on nestings.
    
    Similarly, the set $\cT_\sX(A\not <B)$ is $\max$-tropically convex as it is the intersection of $\cT_\sX$ with the tropical halfspace defined by the inequality $\max_{i,j\in A}D_{ij}\geq\max_{k,\ell\in A\cup B}D_{k\ell}$.
    Remark~\ref{rmk:geomPareto} implies also the co-Pareto property.
\end{proof}

The Pareto property gives a unanimity rule: nestings present in all the trees are also present in the consensus.
One may wonder if this rule can be relaxed as there exist (super)majority-rule consensus trees commonly used for the unweighted case; they are denoted $M_\ell$ by Felsenstein in \cite[Chapter~30]{Felsenstein:2003}.
Indeed, one can find such a rule for tropical medians~\cite{TropMedian}.

\begin{proposition} \label{prop:majRuleTropMedian}
    A nesting appears in the tropical median consensus tree if it appears in a proportion of the input trees greater than $1-1/\binom{n}{2}$.
    Moreover, a nesting will not appear in the tropical median consensus tree if it occurs in a proportion less than $1/\binom{n}{2}$ of the input trees.
\end{proposition}

\begin{proof}
    The tropical median corresponds to the Fermat--Weber problem whose gauge distance is given by the regular simplex.
    Therefore, the essential hull of a finite set $A$ defined in \cite{Comaneci+Plastria} coincides with the $\max$-tropical convex hull of $A$.
    Then the conclusion follows from \cite[Proposition~5.6]{Comaneci+Plastria} and Remark~\ref{rmk:geomPareto}, as in the proof of Proposition~\ref{prop:consensusPareto}.
\end{proof}

\begin{remark}
    Note that a consensus method is not well-defined when there are multiple minimum points.
    Most problematic is the situation when different tree topologies are possible, when it is unclear how to resolve incompatible optimum trees.
    Yet, this is not the case when the set of optimal locations is convex~\cite[Proposition~6]{uglyPaper}: separating the tree space in cones of trees having a tree topology gives rise to a convexly disjoint collection in the sense of \cite[Definition~1.15]{HKT:2009}.

    Nonetheless, the aforementioned proposition applies when the set of all optima in $\torus{\binom{n}{2}}$ is contained in $\cT_\sX$; guaranteed for strictly $\triangle$-star-quasiconvex dissimilarities.
    Otherwise, one might still have problems in defining consistently a consensus method; see \cite[Example~24]{TropMedian} for the symmetric tropical Fermat--Weber problem.
    For this reason, one has to consider the regularized versions discussed in \textsection\ref{subsec:regular}.    
\end{remark}

\section{Conclusion and future perspectives} \label{sec:conclusion}

\noindent
We provided a large class of location estimators whose value lies in the $\max$-tropical convex hull of the input with the purpose of obtaining consensus methods with good properties.
The first direction would be to obtain methods to obtain the optima efficiently.
On the other hand, searching for extra properties of specific location problems could be helpful for applications; more details are provided below.



\subsection{Comparison to consensus methods based on the BHV distance}
We have exploited tropical convexity to obtain consensus methods with good properties.
More precisely, we focused on (co-)Pareto properties that can be interpreted in a purely geometric way.
The associated spaces are also $\max$-tropically convex so the aforesaid properties are immediate for the tropical approach.  

Although the BHV geometry of the tree space is more studied than its tropical counterpart, there are few consensus methods proposed for this geometry.
A first proposal was given in the pioneering paper by Billera, Holmes, and Vogtmann \cite{BHV}, but a few drawbacks were already pointed out: e.g., doubling every input tree changes the output.
An approach based on Fr\'{e}chet means was proposed by Miller et~al.~\cite{MOP:2015} and Ba{\v{c}}{\'a}k~\cite{Bacak:2014}.
It is also Pareto and co-Pareto on splits \cite[Lemma~5.1]{MOP:2015}, but the result is more intricate.
The same properties hold for Fermat--Weber and center problems in the BHV space \cite[Chapter~3]{Botte:2019}.
The approach is again analytical, but similar for all the cases.
One could try a geometric approach, as in the tropical case, as it could lead faster to identification of self-consistent properties for consensus methods.



\subsection{Majority rules in consensus methods}
Proposition~\ref{prop:majRuleTropMedian} provides a supermajority rule for tropical median consensus with respect to nestings.
This can be a step towards understanding the relationship between median weighted trees and the widely used majority-rule consensus for unweigthed trees.
In fact, the majority-rule consensus can be interpreted as a median~\cite{Margush+McMorris}, but it is unclear if this can be extended to weighted phylogenies.

However, Proposition~\ref{prop:majRuleTropMedian} provides a large threshold for a majority rule in the case of tropical median consensus trees, indicating that they are quite conservative.
This seems to be owing to the low breakdown point of the tropical median caused by asymmetry; check \cite{Comaneci+Plastria} for more details.
Therefore, an investigation of location estimators with higher breakdown point could provide a better connection to the majority-rule consensus.




\subsection{Compositional data}

A different application of our location estimators could be to compositional data~\cite{CompData}.
That is, the data can be seen as points in a simplex; our methods would be applied to the centered logratio transform of the input.
Note that $\triangle$-star-quasiconvex sets are defined with respect to special directions, which correspond to the vertices of the simplex.

What is more, the motivation of Tokuyama and Nakano in studying algorithms for transportation problem came from splitting the points from a simplex in multiple regions~\cite{Tokuyama+Nakano:1995}.
Moreover, Nielsen and Sun analyzed clustering methods with the symmetric tropical distance on compositional data showing a better performance than other more commonly used dissimilarity measures~\cite{Nielsen+Sun:2019}.
These results suggest that $\triangle$-star-quasiconvex dissimilarities could be useful in compositional data analysis.

\subsection*{Acknowledgments}
I am indebted to Michael Joswig for discussing different aspects of this paper.
I thank G\"{u}nter Rote for bringing \cite{AHA:1998} to my attention. 
The author was supported by \enquote{Facets of Complexity} (GRK 2434, project-ID 385256563).
\printbibliography

\section*{Appendix A: Convex analysis on $\torus{n}$}

We state and proof a slightly more general form of  Proposition~\ref{prop:exactRegSol} and then we put an Euclidean structure on $\torus{n}$ to show how we can obtain a quantitative result for the regularized version of the tropical Fermat--Weber problem.

\subsection*{The proof of Proposition~\ref{prop:exactRegSol}}

We will prove the result in a finite-dimensional real vector space $X$.
We will equip it with an inner product $\langle\cdot,\cdot\rangle$ which gives an isomorphism $X^*\cong X$.
In this way, we can see the subgradients of a convex function as elements of $X$.
We recall that the subdifferential of a convex function $f:X\to\RR$ at a point $x$ is the set
\[\partial f(x)=\left\{c\in X:f(y)-f(x)\geq\langle c,y-x\rangle\ \forall y\in X\right\}.\]
It will be used to characterize the minima of $f$ through the first-order minimality condition: $x$ is a minimum of $f$ if and only if $\0\in\partial f(x)$.
We refer to the book by Rockafellar~\cite{Rockafellar:70} for more details on convex analysis.

We are interested in optima of regularized versions of $f$ of the form $f+\lambda h$ with $h$ having linear growth.
More specifically, we care of $h$ being \emph{Lipschitz continuous}, i.e. there exists a constant $L>0$ such that $|h(x)-h(y)|\leq L\|x-y\|$ for every $x,y\in X$, where $\|\cdot\|$ is any norm on $X$.\footnote{We assumed that $X$ is finite-dimensional, so every two norms are equivalent. Thus, the definition does not depend on the specific norm. Nevertheless, the constant $L$ depends on $\|\cdot\|$.}

As a last definition, we say that $h$ is polyhedral convex if it is the maximum of finitely many affine functions on $X$.
Now we can state and proof a slight generalization of Proposition~\ref{prop:exactRegSol}.

\begin{proposition} \label{prop:minimaRegLips}
    Let $h:X\to\RR$ be a polyhedral convex function and $f:X\to\RR$ convex and Lipschitz continuous.
    Then there exists a constant $\lambda_0>0$ such that the minima of $h+\lambda f$ are also the minima of $h$ for every $\lambda\in(0,\lambda_0)$.
\end{proposition}

\begin{proof}
    Consider an arbitrary minimum $m_\lambda$ of $h+\lambda f$.
    The first-order optimality condition entails $\0\in\partial h(m_\lambda)+\lambda\partial f(m_\lambda)$.
    What is more, since $f$ is Lipschitz continuous, \cite[Lemma~2.6]{Shalev-Shwartz:2011} yields the existence of a bounded set $B$ such that ${\partial f(x)\subset B}$ for all $x\in X$.

    If $\0\notin\partial h(x)$, then $\0\notin\partial h(x)+\lambda B$ for $\lambda$ sufficiently small, as $\partial h(x)$ is closed.
    We also know that there are finitely many values for $\partial h(x)$, as we assumed $h$ is a polyhedral convex function.
    Accordingly, there exists $\lambda_0>0$ such that $\0\notin\partial h(x)+\lambda B$ for every $\lambda\in(0,\lambda_0)$.
    The last relation implies that $\0\in\partial h(m_\lambda)$ if $\lambda<\lambda_0$, which is equivalent to $m_\lambda$ being a minimum of $h$. 
\end{proof}

\begin{remark} \label{rmk:quantLambda}
    If we know the bounded set $B$ from the proof of Proposition~\ref{prop:minimaRegLips}, then we can set $\lambda_0=\sup\{\lambda>0:\0\notin P+\lambda B,\ \forall P\in\cP\}$ 
    where $\cP$ is the set of all possible values of $\partial h(x)$ such that $\0\notin\partial h(x)$. 
    The infimum is positive, as $\mathcal P$ is a finite collection of closed convex sets.

    If $h$ is a gauge $\gamma$, then \cite[Lemma~2.6]{Shalev-Shwartz:2011} says that we can set $B=\{x\in X:{\gamma^\circ(x)\leq r}\}=:rB_{\gamma^\circ}$ for some $r>0$ where $\gamma^\circ(y):=\sup_{x:\gamma(x)\leq 1}\langle x,y\rangle$ is the dual gauge.
    Hence, $P+\lambda B$ represents the set of points at distance at most $\lambda r$ from $P$ measured by the distance $d_{\gamma^\circ}$ induced from $\gamma^\circ$, i.e. $d_{\gamma^\circ}(x,y)=\gamma^\circ(y-x)$.
    Consequently, we have $\lambda_0=\inf_{P\in\cP}d_{\gamma^\circ}(P,\0)/r$.
\end{remark}

\subsection*{Euclidean structure on $\torus{n}$}

We just conclude with explaining how we can put a Euclidean structure on $\torus{n}$ in a natural way.
The idea is to identify the tropical projective torus with a hyperplane of $\RR^n$ with the regular Euclidean structure.
Using this idea, by factoring with $\RR\1$, one can identify $\torus{n}$ with the orthogonal subspace to $\1$, which is $\cH=\{x\in\RR:x_1+\dots+x_n=0\}$.
This identification is natural as we obtain the same subdifferentials of a convex function $f:\torus{n}\to\RR$ as in the case when we consider it as a function on $\RR^n$ such that $f(x+\lambda\1)=f(x)$ for each $x\in\RR^n$ and $\lambda\in\RR$. 



Having fixed this structure, we search for $\lambda_0$ as in Proposition~\ref{prop:minimaRegLips} for $h(x)=\sum_i \gamma_\infty({x-v_i})$ and $f(x)=\gamma_1(x-v)$ where $v\in\tconv^{\max}(v_1,\dots,v_m)$.
This is, we want quantitative results for regularizations of tropical Fermat--Weber problems.

In this case, the subdifferentials of $h$ are integer polytopes in $\cH$.
Moreover, one can check that the dual gauge of $\gamma_1$ has the expression $\gamma_1^\circ(x)=\gamma_1(-x)/n$ which takes integer values at each point of $\cH\cap\ZZ^n$.
Consequently, $\lambda_0=\inf_{P\in\cP}d_{\gamma_1^\circ}(P,\0)\geq 1$ as it is a positive integer.
Whence, the minima of $h+\lambda f$ are also minima of $h$ for every $\lambda\in(0,1)$.



\end{document}